\documentclass[10pt]{article}
\usepackage{mathrsfs}
\usepackage{amssymb}
\usepackage{amsmath}
\usepackage[all]{xy}
\usepackage{amsthm}
\usepackage{amsfonts}
\usepackage{latexsym}
\usepackage{amscd,amssymb,amsopn,amsmath,amsthm,graphics,amsfonts,mathrsfs,accents,enumerate,verbatim,calc}
\usepackage[dvips]{graphicx}
\usepackage[colorlinks=true,linkcolor=red,citecolor=blue]{hyperref}

\newtheorem{thm}{Theorem}[section]
\newtheorem{prop}[thm]{Proposition}

\newtheorem{lem}[thm]{Lemma}

\newtheorem{cor}[thm]{Corollary}
\newtheorem{df}[thm]{Definition}
\newtheorem{exa}[thm]{Example}

\newtheorem{rem}[thm]{Remark}

\newtheorem{Para}[thm]{}
\newcommand{\A}{\mathscr{A}}
\newcommand{\B}{\mathscr{B}}

\newcommand{\X}{\mathscr{X}}

\newcommand{\C}{\mathscr{C}}
\newcommand{\K}{\mathbb{K}}
\newcommand{\D}{\mathbb{D}}

\def\Aus{\mathop{\rm Aus}\nolimits}

\def\Im{\mathop{\rm Im}\nolimits}
\def\Ker{\mathop{\rm Ker}\nolimits}
\def\Coker{\mathop{\rm Coker}\nolimits}

\def\mod{\mathop{\rm mod}\nolimits}

\def\Hom{\mathop{\rm Hom}\nolimits}
\def\End{\mathop{\rm End}\nolimits}
\def\Ext{\mathop{\rm Ext}\nolimits}
\def\Tor{\mathop{\rm Tor}\nolimits}

\def\pd{\mathop{\rm pd}\nolimits}

\def\Gpd{\mathop{\rm Gpd}\nolimits}

\def\proj{\mathop{\rm proj}\nolimits}

\def\Gproj{\mathop{\rm Gproj}\nolimits}

\def\Gpd{\mathop{\rm Gpd}\nolimits}
\def\inf{\mathop{\rm inf}\nolimits}

\def\Con{\mathop{\rm Con}\nolimits}
\def\Im{\mathop{\rm Im}\nolimits}
\begin{document}

\title{\large \bf When the Schur functor induces a triangle-equivalence between Gorenstein defect categories
}
\author{Huanhuan Li$^a$, Jiangsheng Hu$^b$ and Yuefei Zheng$^{c}$\footnote{Corresponding author} \\
\it\footnotesize $^a$School of Mathematics and Statistics, Xidian University, Xi'an 710071, Shaanxi Province, China\\
\it\footnotesize $^b$School of Mathematics and Physics, Jiangsu University of Technology, Changzhou 213001, Jiangsu Province, China\\
\it\footnotesize $^c$College of Science, Northwest A$\&$F University, Yangling 712100, Shaanxi Province, China\\
\it\footnotesize Email addresses: lihh@xidian.edu.cn, jiangshenghu@jsut.edu.cn and zhengyf@nwafu.edu.cn
}
\date{}
\baselineskip=14pt
\maketitle
\begin{abstract} Let $R$ be an Artin algebra and $e$ an idempotent of $R$. Assume that $\Tor_i^{eRe}(Re,G)=0$ for any $G\in\Gproj eRe$ and $i$ sufficiently large. Necessary and sufficient conditions are given for the Schur functor $S_e$ to induce a triangle-equivalence $\D_{def}(R)\simeq\D_{def}(eRe)$. Combine this with a result of Psaroudakis-Skartsaterhagen-Solberg \cite{PSS}, we provide necessary and sufficient conditions for the singular equivalence $\D_{sg}(R)\simeq\D_{sg}(eRe)$ to restrict to a triangle-equivalence $\underline{\Gproj R}\simeq\underline{\Gproj eRe}$. Applying these to the triangular matrix algebra $T=\left(
                                                                                 \begin{array}{cc}
                                                                                   A & M \\
                                                                                   0 & B \\
                                                                                 \end{array}
                                                                               \right)$, corresponding results between candidate categories of $T$ and $A$ (resp. $B$) are obtained. As a consequence, we infer Gorensteinness and CM-freeness of $T$ from those of $A$ (resp. $B$). Some concrete examples are given to indicate one can realise the Gorenstein defect category of a triangular matrix algebra as the singularity category of one of its corner algabras. 
\end{abstract}

%

\section{\bf Introduction}

Let $R$ be an Artin algebra. In the study of stable homological algebra and Tate cohomology, Buchweitz \cite{Bu} introduced the singularity category $\D_{sg}(R)$ of $R$. This category is a certain Verdier quotient of the bounded derived category of $\mod R$ by modulo the bounded homotopy category of $\proj R$, where $\mod R$ denotes the category of finitely generated left $R$-modules and $\proj R$ is its subcategory consisting of projective modules.
Later on, this category was reconsidered by Orlov \cite{O} in the setting of algebraic geometry and turned out to have a closed relation with the ``Homological Mirror Symmetry Conjecture''.
The singularity category of $R$ measures the ``regularity'' of $R$ in sense that $\D_{sg}(R)=0$ if and only if $R$ is of finite global dimension.

By the fundamental result in \cite{Bu}, the stable category $\underline{\Gproj R}$ of finitely generated Gorenstein projective modules might be regarded as a triangulated subcategory of $\D_{sg}(R)$ via an triangulated embedding functor $F$. Besides, $F:\underline{\Gproj R}\to\D_{sg}(R)$ is a triangle-equivalence provided that $R$ is Gorenstein \cite{Bu,H2}. Inspired by this, Bergh, J{\o}rgensen and Oppermann (\cite{BJO}) considered the Verdier quotient $\D_{def}(R):=\D_{sg}(R)/\Im F$, and they called it the Gorenstein defect category of $R$.
This category measures how far the algebra $R$ is from being Gorenstein. More precisely, $R$ is Gorenstein if and only if $\D_{def}(R)$ is trivial.
Nowadays, singularity categories and related topics have been studied by many authors, see for example \cite{C1,C2,CZ,LL,L,R1,ZP}.

It is of interest to consider when two Artin algebras share the same singularity category up to a triangle-equivalence. In this case, we call these two algebras singular equivalent and such an equivalence is called a singular equivalence. There has been a lot of work by many people to find certain conditions when two algebras are singular equivalent \cite{C1,C3,C4,PSS}.
Referring to Buchweitz's work (\cite[Theorem 4.4.1]{Bu}), it is natural to ask: when a singular equivalence restricts to a triangle-equivalence between stable categories of Gorenstein projective modules? Unfortunately, the answer of the questions is unknown in general.
However, in view that the Gorenstein defect category is a Verdier quotient of the singularity category by modulo the isomorphic image of stable category of Gorenstein projective modules.
It is not hard to see one necessary condition for the establishment of the above question is the existence of a triangle-equivalence between their Gorenstein defect categories.
Besides, due to Kong-Zhang \cite{KZ}, the Gorenstein defect category of a CM-finite algebra is equivalent to the singularity category of its relative Auslander algebra. So for two CM-finite algebras, a triangle-equivalence between their Gorenstein defect categories coincides with a singular equivalence of their relative Auslander algebras. Meanwhile, the study of triangle-equivalences between Gorenstein defect categories is useful for some special algebras. For instance, let $A$ be a connected Artin algebra with radical square zero. Following Chen's work (\cite{C2'}), $A$ is either self-injective or CM-free. So if another algebra $B$ has the same Gorenstein defect category with $A$ (up to a triangle-equivalence), then either $B$ is Gorenstein or $B$ has $\D_{sg}(A)$ as its Gorenstein defect category. This provides an effective way to detect the Gorensteinness and CM-freeness of some certain triangular matrix algebra from those of its corner algebras (see Theorems \ref{thm:1.3} and \ref{thm:1.2}).
For the mentioned reasons, it is necessary to study when two algebras have the same Gorenstein defect category.

Let $R$ be an Artin algebra and $e\in R$ an idempotent. The Schur functor associative to $e$ is defined to be $S_e=eR\otimes_R-:\mod R\to \mod eRe$ (\cite{G}).
The idempotent $e$ is said to be {\it singularly-complete} if the projective dimension of any $R/ReR$-module is finite as an $R$-module.
It was shown in \cite{C1} that if $e$ is singularly-complete and the projective dimension $\pd_{eRe}eR<\infty$, then the Schur functor induces a singular equivalence $\D_{sg}(R)\simeq\D_{sg}(eRe)$. Later on, this result was generalized by Psaroudakis-Skartsaterhagen-Solberg \cite{PSS}, who gave necessary and sufficient conditions for such singular equivalence by using the technique of recollements of abelian categories. Thanks to these beautiful work, our main result (Theorem \ref{thm:1.1}) gives necessary and sufficient conditions for the Schur functor to induce a triangle-equivalence $\D_{def}(R)\simeq\D_{def}(eRe)$.
As a consequence, we give an affirmative answer for the singular equivalence $\D_{sg}(R)\simeq\D_{sg}(eRe)$ to restrict to a triangle-equivalence $\underline{\Gproj R}\simeq\underline{\Gproj eRe}$ (see Corollary \ref{cor:3.3}). Then some applications in the triangular matrix algebra and some concrete examples are given. The outline of this paper is as follows.

In Section \ref{Gdc}, we deal with the subcategory $\D^b(\mod R)_{fgp}$ of $\D^b(\mod R)$ consisting of complexes with finite Gorenstein projective dimension. Use it, we give a description of the stable category $\underline{\Gproj R}$ of Gorenstein projective modules and the Gorenstein defect category $\D_{def}(R)$ (see Theorem \ref{thm:A.5}). Consequently, the converse of Buchweitz's Theorem is proved.

Let $e$ be an idempotent of $R$ and $\Tor_i^{eRe}(Re,G)=0$ for any $G\in\Gproj eRe$ and $i$ sufficiently large. In Section \ref{Gde}, necessary and sufficient conditions are given for the Schur functor $S_e$ to induce a triangle-equivalence $\D_{def}(R)\simeq\D_{def}(eRe)$ (see Theorem \ref{thm:1.1}). Combine this with \cite[Corollary 5.4]{PSS}, we obtain necessary and sufficient conditions to
get an exact commutative diagram $$\xymatrix{0\ar[r]&\underline{\Gproj R}\ar[r]\ar[d]
&\D_{sg}(R)\ar[r]\ar[d]
& \D_{def}(R)\ar[d]\ar[r]&0\\
0\ar[r]&\underline{\Gproj eRe}\ar[r]
&\D_{sg}(eRe)\ar[r]
&\D_{def}(eRe)\ar[r]&0
}$$
with all vertical functors triangle-equivalences (see Corollary \ref{cor:3.3}).

In Section \ref{app}, we apply Theorem \ref{thm:1.1} and Corollary \ref{cor:3.3} to the triangular matrix algebra $T=\left(
                                                                                 \begin{array}{cc}
                                                                                   A & M \\
                                                                                   0 & B \\
                                                                                 \end{array}
                                                                               \right)$, where the corner algebras $A$ and $B$ are Artin algebras and $_AM_B$ is an $A$-$B$-bimodule. To take $R=T$ and $e=e_A=\left(
                                                                                                                                                                                                                \begin{array}{cc}
                                                                                                                                                                                                                  1 & 0\\
                                                                                                                                                                                                                  0 & 0 \\
                                                                                                                                                                                                                \end{array}
                                                                                                                                                                                                              \right)$ (resp. $e=e_B=\left(
                                                                                                                                                                                                                \begin{array}{cc}
                                                                                                                                                                                                                  0 & 0\\
                                                                                                                                                                                                                  0 & 1 \\
                                                                                                                                                                                                                \end{array}
                                                                                                                                                                                                              \right)$) as in Theorem \ref{thm:1.1} and Corollary \ref{cor:3.3},
we get the corresponding results between candidate categories of $T$ and $A$ (resp. $B$), see Theorems \ref{thm:1.3} and \ref{thm:1.2} for details. As a consequence, we could infer Gorensteinness (resp. CM-freeness) of the triangular matrix algebra $T$ from that of its corner algebra $A$ (resp. $B$). Finally, some explicit examples are given to indicate one can realise the Gorenstein defect category of a triangular matrix algebra as the singularity category of one of the corner algebras.

\section{Gorenstein defect categories}\label{Gdc}

Throughout, all algebras are Artin algebras over some commutative Artinian ring and all modules are finitely generated. For a given algebra $R$,
denote by $\mod R$ the category of left R-modules; right $R$-modules are viewed as left $R^{op}$-modules, where $R^{op}$ is the opposite algebra of $R$.
We use $\proj R$ to denote the subcategories of $\mod R$ consisting of projective modules. The $*$-bounded derived category of $\mod R$ and homotopy category of $\proj R$
will be denoted by $\D^*(\mod R)$ and $\K^*(\proj R)$ respectively, where $*\in\{blank,\ +,\ -, \ b\}$.

Usually, we use $_RM$ (resp. $M_R$) to denote a left (resp. right) $R$-module $M$, and the projective dimension of $_RM$ (resp. $M_R$) will be denoted by $\pd_RM$ (resp. $\pd M_R$).
For a subclass $\X$ of $\mod R$. Denote by $\X^\bot$ (resp. $^\bot\X$) the subcategory consisting of modules $M\in \mod R$ such that $\Ext_R^1(X,M)=0$ (resp. $\Ext_R^1(M,X)=0$) for any $X\in\X$.

\vspace{0.2cm}

Let $$X^\bullet=\cdots\to X^{-1}\xrightarrow{d^{-1}}X^{0}\xrightarrow{d^{0}} X^{1}\to\cdots$$
be a complex in $\mod R$. For any integer $n$, we set $Z^n(X^\bullet)=\Ker d^n$, $B^n(X^\bullet)=\Im d^{n-1}$ and $H^n(X^\bullet)=Z^n(X^\bullet)/B^n(X^\bullet)$. $X^{\bullet}$ is called  acyclic (or exact) if $H^{n}(X^{\bullet})=0$ for any $n\in \mathbb{Z}$.

Recall from \cite{AB,AM,Ho} that an acyclic complex $X^\bullet$ is called {\it totally acyclic} if each $X^i\in\proj R$ and $\Hom_R(X^\bullet,R)$ is acyclic. A module $M\in\mod R$ is {\it Gorenstein projective} if there exists some totally acyclic complex $X^\bullet$ such that $M\cong Z^0(X^\bullet)$.
Denote by $\Gproj R$ the subcategory of $\mod R$ consisting of Gorenstein projective modules.
Given a module $M\in\mod R$, the {\it Gorenstein projective dimension} $\Gpd_RM$ of $M$ is defined to be $\Gpd_RM=\inf\{n:$ there exists an exact sequence $0\to G_n\to \cdots\to G_1\to G_0\to M\to0$, where each $G_i\in\Gproj R$$\}$.

%
%

It is well-known that $\Gproj R$ is a Frobenius category, and hence its stable category $\underline{\Gproj R}$ is a triangulated category (\cite{H1}).
Consider the composition of the embedding functor $\Gproj R\hookrightarrow\D^b(\mod R)$ and the localization functor $\D^b(\mod R)\to\D_{sg}(R)$. It induces a functor $F:\underline{\Gproj R}\to\D_{sg}(R)$, which sends every Gorenstein projective module to the stalk complex concentrated in degree zero.

\begin{lem}(Buchweitz's Theorem, see \cite[Theorem 4.4.1]{Bu}) Keep the above notations. The canonical functor $F:\underline{\Gproj R}\to\D_{sg}(R)$ is an embedding triangle-functor. Furthermore, $F$ is a triangle-equivalence provided that $R$ is Gorenstein (that is, the left and right self-injective dimensions of $R$ are finite).
\end{lem}

According to Buchweitz's Theorem, $\Im F$ is a triangulated subcategory of $\D_{sg}(R)$.

\begin{df}\cite{BJO} {\rm We call the Verdier quotient $\D_{def}(R):=\D_{sg}(R)/\Im F$ the {\it Gorenstein defect category} of $R$.
}\end{df}

Recall from \cite{C2'} that $R$ is called {\it CM-free} if $\Gproj R=\proj R$. While $R$ is called {\it CM-finite} (\cite{Be2}) if $\Gproj R$ is of finite representation type. Let $R$ be CM-finite and $\{G_1,G_2,\cdots,G_n\}$ the
set of all pairwise non-isomorphic indecomposable Gorenstein projective modules. Recall from \cite{KZ} that the opposite of the endomorphism algebra $\Aus(R):={\End(\bigoplus\limits_{1\leq i\leq n}G_i)}^{op}$
is called the {\it relative Auslander algebra} of $R$.

\begin{rem}\label{rem:2.4} {\rm
\begin{enumerate}
\item[]
\item[(1)] It is not hard to see $R$ is CM-free if and only if $\D_{def}(R)=\D_{sg}(R)$.
\item[(2)] Following \cite[Corollary 6.10]{KZ}, if $R$ is CM-finite, then there is a triangle-equivalence $\D_{def}(R)\simeq\D_{sg}(\Aus(R))$.
\end{enumerate}
}
\end{rem}


Recently, the Gorenstein defect category $\D_{def}(R)$ was reconsidered by Kong and Zhang \cite{KZ}. What's more, they found $\D_{def}(R)$ is triangle equivalent to $\D^b(\mod R)/\langle\Gproj R\rangle$, where $\langle\Gproj R\rangle$ is the triangulated subcategory of $\D^b(\mod R)$ generated by $\Gproj R$. We wonder what the exact form of objects in $\langle\Gproj R\rangle$. We introduce the following.

\begin{df}\label{df:2.4} {\rm A complex $X^\bullet\in\D^b(\mod R)$ is said to have {\it finite Gorenstein projective dimension} if $X^\bullet$ is isomorphic to some bounded complex consisting of Gorenstein projective modules in $\D^b(\mod R)$.
}\end{df}

We note that the finiteness of Gorenstein projective dimension for a complex in Definition \ref{df:2.4} coincides with that of \cite{Vel}, see for example \cite[Construction 5.5]{Vel}. Denote by $\D^b(\mod R)_{fgp}$ the subcategory of $\D^b(\mod R)$ consisting of complexes with finite Gorenstein projective dimension.

%

\begin{lem}\label{lem:2.5} Let $X^\bullet\in\D^b(\mod R)$. Then the following are equivalent:
\begin{enumerate}
\item[(1)] $X^\bullet\in\D^b(\mod R)_{fgp}$.
\item[(2)]For any quasi-isomorphism $P^\bullet\to X^\bullet$ with $P^\bullet\in\K^{-,b}({\proj R})$, one has $Z^{i}(P^\bullet)\in\Gproj R$ for $i\ll0$, where $\K^{-,b}(\proj R)$ is the full subcategory of $\K^{-}(\proj R)$ consisting of complexes with finite nonzero homology.
\item[(3)] There exists a quasi-isomorphism $P^\bullet\to X^\bullet$ with $P^\bullet\in\K^{-,b}({\proj R})$ such that $Z^{i}(P^\bullet)\in\Gproj R$ for $i\ll0$.
\end{enumerate}
\end{lem}
\begin{proof} (1)$\Longrightarrow$(2). Let $P^\bullet\to X^\bullet$ be a quasi-isomorphism with $P^\bullet\in\K^{-,b}({\proj R})$. Since $X^\bullet\in\D^b(\mod R)_{fgp}$, we get $X^\bullet\cong G^\bullet$ in $\D^b(\mod R)$ with $G^\bullet$ a bounded complex consisting of Gorenstein projective modules. It follows that $P^\bullet\cong G^\bullet$ in $\D^b(\mod R)$ and then there is a quasi-isomorphism $f:P^\bullet\to G^\bullet$ by \cite[1.4.P]{AFHd}. Hence the mapping cone $\Con(f)=\cdots\to P^{-n-1}\to P^{-n}\to P^{-n+1}\oplus G^{-n}\to\cdots$
is acyclic. Note that $\Con(f)$ is bounded above with each degree in $\Gproj R$ and $\Gproj R$ is closed under kernels of epimorphisms (see \cite[Theorem 2.5]{Ho}). We conclude that $Z^{i}(P^\bullet)\cong Z^{i-1}(\Con(f))\in\Gproj R$ for $i\ll0$.

(2)$\Longrightarrow$(3) is trivial.

(3)$\Longrightarrow$(1). Let $P^\bullet\to X^\bullet$ be a quasi-isomorphism with $P^\bullet\in\K^{-,b}({\proj R})$ and $Z^{i}(P^\bullet)\in\Gproj R$ for $i\ll0$. Since $P^\bullet\in\K^{-,b}({\proj R})$, $P^\bullet$ is isomorphic to the following complex in $\D^b(\mod R)$
$$G^\bullet:=0\to Z^{t}(P^\bullet)\to P^{t}\to\cdots\to P^{s-1}\to P^s\to0,$$
where $s$ is the supremum of index $i\in\mathbb{Z}$ such that $P^i\neq0$ and $t$ is the index such that $Z^{i}(P^\bullet)\in\Gproj R$ and $H^i(P^\bullet)=0$ for any $i\leq t$. Hence $X^\bullet\cong G^\bullet$ in $\D^b(\mod R)$ with $G^\bullet$ a bounded complex consisting of Gorenstein projective modules and then $X^\bullet\in\D^b(\mod R)_{fgp}$.
\end{proof}

\begin{rem}\label{rem:2.7}{\rm It follows from Lemma \ref{lem:2.5} and \cite[Theorem 3.1]{AM} that an $R$-module $M$ viewed as a stalk complex concentrated in degree zero lies in $\D^b(\mod R)_{fgp}$ if and only if $\Gpd_RM<\infty$.
}\end{rem}

Let $X^\bullet$ be a complex of $R$-modules. The {\it length} $l(X^\bullet)$ of $X^\bullet$ is defined to be the cardinal of the set $\{X^i\neq0|i\in\mathbb{Z}\}$. Let $n\in\mathbb{Z}$, denote by $X^\bullet_{\geqslant n}$ the complex with the $i$th component equal to $X^i$ whenever $i\geqslant n$ and to 0 elsewhere.
\begin{thm}\label{thm:A.3} $\D^b(\mod R)_{fgp}$ is a thick subcategory of $\D^b(\mod R)$.  Furthermore, $\D^b(\mod R)_{fgp}=\langle \Gproj R\rangle$.
\end{thm}
\begin{proof}  It follows from \cite[Proposition 3.2]{ZH} that $\D^b(\mod R)_{fgp}$ is a triangulated subcategory of $\D^b(\mod R)$.  To get the first assertion, it suffices to show $\D^b(\mod R)_{fgp}$ is closed under direct
summands.  In fact, let $X_1^\bullet\oplus X_2^\bullet\in\D^b(\mod R)_{fgp}$ with $X_1^\bullet, X_2^\bullet\in\D^b(\mod R)$. Choose quasi-isomorphisms $P_1^\bullet\to X_1^\bullet$ and $P_2^\bullet\to X_2^\bullet$ with $P_1^\bullet, P_2^\bullet\in\K^{-,b}({\proj R})$. It follows that $P_1^\bullet\oplus P_2^\bullet\to X_1^\bullet\oplus X_2^\bullet$ is a quasi-isomorphism. Notice that $X_1^\bullet\oplus X_2^\bullet\in\D^b(\mod R)_{fgp}$ and $P_1^\bullet\oplus P_2^\bullet\in\K^{-,b}({\proj R})$, it follows from Lemma \ref{lem:2.5} that $Z^{i}(P_1^\bullet\oplus P_2^\bullet)\in\Gproj R$ for $i\ll0$. Since $Z^{i}(P_1^\bullet\oplus P_2^\bullet)\cong Z^{i}(P_1^\bullet)\oplus Z^{i}(P_2^\bullet)$, we get $Z^{i}(P_1^\bullet), Z^{i}(P_2^\bullet)\in\Gproj R$ for $i\ll0$. Then by Lemma \ref{lem:2.5}, we get $X_1^\bullet, X_2^\bullet\in\D^b(\mod R)_{fgp}$.

Note that every Gorenstein projective module viewed as a stalk complex concentrated in degree zero has finite Gorenstein projective dimension. Thus $\Gproj R\subseteq\D^b(\mod R)_{fgp}$
and then $\langle \Gproj R\rangle\subseteq\D^b(\mod R)_{fgp}$. On the other hand, let $G^\bullet$ be a bounded complex consisting of Gorenstein projective $R$-modules. We will show $G^\bullet\in\langle\Gproj R\rangle$ to complete the proof. We proceed by induction on the length $l(G^\bullet)$ of $G^\bullet$. If $l(G^\bullet)=1$, it is trivial to verify $G^\bullet\in\langle \Gproj R\rangle$. Now let $l(G^\bullet)=n\geq2$. We may suppose $G^m\neq0$ and $G^i=0$ for $i<m$. The we have the following triangle in $\D^b(\mod R)$:
$$G^m[-m-1]\to G^\bullet_{\geq{m+1}}\to G^\bullet\to G^m[-m].$$
By the induction hypothesis, we have that both $G^m[-m-1]$ and $G^\bullet_{\geq{m+1}}$ lie in $\langle \Gproj R\rangle$. Therefore $G^\bullet\in\langle \Gproj R\rangle$.
\end{proof}

\begin{prop}\label{prop:A.4} Let $X^\bullet\in\D^b(\mod R)$. If each $X^i$ is of finite Gorenstein projective dimension as an $R$-module, then $X^\bullet\in\D^b(\mod R)_{fgp}$. Furthermore, $\D^b(\mod R)_{fgp}=\D^b(\mod R)$ if and only if $R$ is Gorenstein.
\end{prop}

\begin{proof} We will proceed by induction on the length $l(X^\bullet)$ of $X^\bullet$. If $l(X^\bullet)=1$, it is trivial to verify $X^\bullet\in\D^b(\mod R)_{fgp}$. Now let $l(X^\bullet)=n\geq2$. We may suppose $X^m\neq0$ and $X^i=0$ for $i<m$. Then we have the following triangle in $\D^b(\mod R)$:
$$X^m[-m-1]\to X^\bullet_{\geq{m+1}}\to X^\bullet\to X^m[-m].$$
By the induction hypothesis, we have that both $X^m[-m-1]$ and $X^\bullet_{\geq{m+1}}$ lie in $\D^b(\mod R)_{fgp}$. Therefore $X^\bullet\in\D^b(\mod R)_{fgp}$.

Note that $R$ is Gorenstein if and only if every module in $\mod R$ has finite Gorenstein projective dimension by \cite[Theorem]{Hos}. Thus $\D^b(\mod R)_{fgp}=\D^b(\mod R)$ if and only if $R$ is Gorenstein.
\end{proof}

The following result seems clear (see e.g. \cite[Theorem 3.4]{ZH}), we provide a proof here.

\begin{thm}\label{thm:A.5} We have the following exact commutative diagram:
$$\xymatrix{0\ar[r]&\underline{\Gproj R}\ar[r]^F\ar[d]
&\D_{sg}(R)\ar[r]\ar@{=}[d]
& \D_{def}(R)\ar[d]\ar[r]&0\\
0\ar[r]&{\D^b(\mod R)_{fgp}/\K^b(\proj R)}\ar[r]
&\D_{sg}(R)\ar[r]
&{\D^b(\mod R)/\D^b(\mod R)_{fgp}}\ar[r]&0
}$$
with all vertical functors triangle-equivalences.
\end{thm}
\begin{proof} In view of Buchweitz's Theorem, it suffices to show $\Im F=\D^b(\mod R)_{fgp}/\K^b(\proj R)$.

Note that every Gorenstein projective module viewed as a stalk complex concentrated in degree zero has finite Gorenstein projective dimension. Thus
$$\Im F\subseteq\D^b(\mod R)_{fgp}/\K^b(\proj R).$$
Now let $X^\bullet\in\D^b(\mod R)_{fgp}/\K^b(\proj R)$, it follows from Lemma \ref{lem:2.5} that there exists a quasi-isomorphism $P^\bullet\to X^\bullet$ with $P^\bullet\in\K^{-,b}({\proj R})$ such that $Z^{i}(P^\bullet)\in\Gproj R$ for $i\ll0$.
Hence $P^\bullet$ is isomorphic to the following complex in $\D^b(\mod R)$
$$G^\bullet:=0\to Z^{t}(P^\bullet)\to P^{t}\to\cdots\to P^{s-1}\to P^s\to0,$$
where $s$ is the supremum of index $i\in\mathbb{Z}$ such that $P^i\neq0$ and $t$ is the index such that $Z^{i}(P^\bullet)\in\Gproj R$ and $H^i(P^\bullet)=0$ for any $i\leq t$.
 Consider the following triangle in $\D_{sg}(\mod R)$:
$$Z^{t}(P^\bullet)[-t]\to P^\bullet_{\geq t}\to G^\bullet\to Z^{t}(P^\bullet)[-t+1].$$
Since $P^\bullet_{\geq t}\in\K^b(\proj R)$, $G^\bullet\cong Z^{t}(P^\bullet)[-t+1]$ and then $X^\bullet\cong Z^{t}(P^\bullet)[-t+1]$ in $\D_{sg}(\mod R)$.
Hence $X^\bullet[t-1]\cong Z^{t}(P^\bullet)$, that is, $X^\bullet[t-1]\in\Im F$. Since $\Im F$ is a triangulated subcategory, $X^\bullet\in\Im F$
and then $\D^b(\mod R)_{fgp}/\K^b(\proj R)\subseteq\Im F.$
\end{proof}

Thanks to Beligiannis \cite{Be1}, Bergh-J{\o}rgensen-Oppermann \cite{BJO}, Kong-Zhang \cite{KZ} and Zhu \cite{Z}, the converse of Buchweitz's Theorem also holds true. We obtain the following.

\begin{cor}\label{cor:A.6} The following are equivalent:
\begin{enumerate}
\item[(1)] $F:\underline{\Gproj R}\to\D_{sg}(R)$ is a triangle-equivalence.
\item[(2)] $\D_{def}(R)=0$.
\item[(3)] $R$ is Gorenstein.
\end{enumerate}
\end{cor}
\begin{proof} (1)$\Leftrightarrow$(2) is trivial, and (2)$\Leftrightarrow$(3) follows from Proposition \ref{prop:A.4} and Theorem \ref{thm:A.5}.
\end{proof}
%
%

\section{Triangle-equivalence of Gorenstein defect categories induced by the Schur functor}\label{Gde}

In this section, let $R$ be an Artin algebra and e an idempotent of $R$. Recall from \cite[Chapter 6]{G} that the Schur functor associative to $e$ is defined to be
$$S_e=eR\otimes_R-:\mod R\to \mod eRe,$$
also it was called the restriction functor in \cite[I.6]{ASS}. Clearly, $S_e$ admits a fully faithful left adjoint $$T_e=Re\otimes_{eRe}-:\mod eRe\to\mod R$$
and a fully faithful right adjoint
$$L_e=\Hom_{eRe}(eR,-):\mod eRe\to\mod R.$$

Recall from \cite{P,PSS} that a {\it recollement} between abelian categories $\A$, $\B$ and $\C$ is a diagram
$$\xymatrix{\A\ar[r]^{i_*}&\B\ar@/^1pc/[l]^{i^!}\ar@/_1pc/[l]_{i^*}\ar[r]^{j^*} &\C\ar@/^1pc/[l]^{j_*}\ar@/_1pc/[l]_{j_!} ,}$$
satisfying the following conditions:
\begin{enumerate}
\item[(1)] $(i^*,i_*,i^!)$ and $(j_!,j^*,j_*)$ are adjoint triples.
\item[(2)] The functors $i_*$, $j_!$ and $j_*$ are fully faithful.
\item[(3)] $\Im i_*=\Ker j^*$.
\end{enumerate}

\vspace{0.2cm}

We need the following fact.

\begin{lem} \cite{P,PSS}\label{lem:3.1} We have the following recollement between module categories:
$$\xymatrix@=2cm{\mod R/ReR\ar[r]^{inc}& \mod R\ar@/^1pc/[l]^{\Hom_R(R/ReR,-)}\ar@/_1pc/[l]_{R/ReR\otimes_R-}\ar[r]^{S_e=eR\otimes_R-} & \mod eRe\ar@/^1pc/[l]^{L_e=\Hom_{eRe}(eR,-)}\ar@/_1pc/[l]_{T_e=Re\otimes_{eRe}-},}$$
where $inc: \mod R/ReR\to\mod R$ denotes the inclusion functor induced by the canonical ring homomorphism $R\to R/ReR$.
In the following, the image of the functor $inc:\mod R/ReR\to\mod R$ will be identified with $\mod R/ReR$ for simplicity.
\end{lem}

The idempotent $e\in R$ is said to be {\it singularly-complete} (\cite{C1}) if the projective dimension of any $R/ReR$-module is finite as an $R$-module.
Chen has shown in \cite{C1} that if $e$ is singularly-complete and $\pd_{eRe}eR<\infty$, then $S_e$ induces a singular equivalence $\D_{sg}(R)\simeq\D_{sg}(eRe)$.
This result is generalized by Psaroudakis-Skartsaterhagen-Solberg \cite{PSS}, where the authors prove $S_e$ induces a singular equivalence $\D_{sg}(R)\simeq\D_{sg}(eRe)$ if and only if $e$ is singularly-complete and $\pd_{eRe}eR<\infty$. Inspired by these, we will consider when $S_e$ induces a triangle-equivalence of Gorenstein defect categories $\D_{def}(R)\simeq\D_{def}(eRe)$.
We call the idempotent $e\in R$ {\it Gorenstein singularly-complete} if the Gorenstein projective dimension of any $R/ReR$-module is finite as an $R$-module. Clearly, $e$ is Gorenstein singularly-complete provided that $e$ is singularly-complete.

Since $S_e$ is exact, it lifts to the $*$-bounded derived functors $\D^*(S_e):\D^*(\mod R)\to\D^*(\mod eRe)$ via $\D^*(S_e)(X^\bullet)=S_e(X^\bullet)$ for each complex $X^\bullet\in\D^*(\mod R)$, where $*\in\{blank,\ +,\ -, \ b\}$. Meanwhile, it is not hard to see $T_e$ preserves projectives. So $T_e$ lifts to a derived functor $\D^-(T_e):\D^-(\mod eRe)\to\D^-(\mod R)$, compare \cite[Proposition 2.1]{CPS2}.

\begin{lem}\label{lem:3.2} The following statements hold true:
\begin{enumerate}
\item[(1)] The $*$-bounded derived functor $\D^*(S_e):\D^*(\mod R)\to\D^*(\mod eRe)$ restricts to
$$\D^b(S_e)_{fgp}:\D^b(\mod R)_{fgp}\to \D^b(\mod eRe)_{fgp}$$ if and only if $\Gpd_{eRe}S_e(F)<\infty$ for any $F\in\Gproj R$.
\item[(2)] Assume that $\Gpd_{R}T_e(G)<\infty$ for any $G\in\Gproj eRe$. Then the derived functor $\D^-(T_e):\D^-(\mod eRe)\to\D^-(\mod R)$ restricts to $$\D^b(T_e)_{fgp}:\D^b(\mod eRe)_{fgp}\to \D^b(\mod R)_{fgp}$$ if and only if $\Tor_i^{eRe}(Re,G)=0$ for any $G\in\Gproj eRe$ and $i$ sufficiently large.
\end{enumerate}
\end{lem}

\begin{proof} (1) Assume that $\D^*(S_e):\D^*(\mod R)\to\D^*(\mod eRe)$ restricts to
$\D^b(S_e)_{fgp}:\D^b(\mod R)_{fgp}\to \D^b(\mod eRe)_{fgp}$. For any $F\in\Gproj R$, we have that $\D^b(S_e)(F)=S_e(F)$ lies in $\D^b(\mod eRe)_{fgp}$. By Remark \ref{rem:2.7}, we get $\Gpd_{eRe}S_e(F)<\infty$. Conversely, assume that $\Gpd_{eRe}S_e(F)<\infty$ for any $F\in\Gproj R$. Let $X^\bullet$ be a bounded complex of Gorenstein projective $R$-modules. It follows that $\D^*(S_e)(X^\bullet)=S_e(X^\bullet)$, it is a bounded complex with each degree being of finite Gorenstein projective dimension. Hence $\D^*(S_e)(X^\bullet)\in\D^b(\mod eRe)_{fgp}$ by Proposition \ref{prop:A.4}. Thus $\D^*(S_e):\D^*(\mod R)\to\D^*(\mod eRe)$ restricts to
$\D^b(S_e)_{fgp}:\D^b(\mod R)_{fgp}\to \D^b(\mod eRe)_{fgp}$.

(2) For the ``if'' part, we will show $\D^-(T_e)(\D^b(\mod eRe)_{fgp})\subseteq\D^b(\mod R)_{fgp}$. To do this it suffices to show $\D^-(T_e)(Y^\bullet)\in\D^b(\mod R)_{fgp}$ for any bounded complex $Y^\bullet$ of Gorenstein projective $eRe$-modules. We proceed by induction on the length $l(Y^\bullet)$ of $Y^\bullet$.

If $l(Y^\bullet)=1$, we may suppose $Y^\bullet=Y$ be the stalk complex concentrated in degree 0.
Take a projective resolution $\cdots\to P^{-n}\to\cdots\to P^{-1}\to P^0\to Y\to0$ of $Y$, and denote by $P^\bullet=\cdots\to P^{-n}\to\cdots\to P^{-1}\to P^0\to0$. It follows that $\D^-(T_e)(Y)\cong T_e(P^\bullet)=Re\otimes_{eRe}P^\bullet$, it is a complex of projective $R$-modules. Since $Y\in\Gproj eRe$, one has each cycle $Z^i(P^\bullet)\in\Gproj eRe$. Then by assumption we obtain $\Gpd_{R}T_e(Z^i(P^\bullet))<\infty$ for every integer $i$. Note that $\Tor_i^{eRe}(Re,G)=0$ for any $G\in\Gproj eRe$ and $i$ sufficiently large. It follows that $T_e(P^\bullet)$ is exact in degree $i$ and hence $Z^i(T_e(P^\bullet))\cong T_e(Z^i(P^\bullet))$ whenever $i\ll0$. Thus there exists some integer $n_0\gg0$ such that $T_e(P^\bullet)$ is isomorphic to its truncation complex $G^\bullet:=0\to T_e(Z^{-n_0}(P^\bullet))\to T_e(P^{-n_0})\to\cdots\to T_e(P^{-1})\to T_e(P^0)\to0$. Note that $T_e(P^{i})\in\proj R$ for every integer $i$ and $\Gpd_{R}T_e(Z^{-n_0}(P^\bullet))<\infty$. Therefore, by Proposition \ref{prop:A.4} we have $\D^-(T_e)(Y)\in\D^b(\mod R)_{fgp}$.

Now suppose $l(Y^\bullet)=n\geq2$ and the claim holds true for any integer less than $n$. Then $Y^\bullet$ must be of the following form:
$$Y^\bullet=0\to Y^{m+1}\to Y^{m+2}\to\cdots\to Y^{m+n}\to0.$$
It induces a triangle $$Y^{m+1}[-m-2]\to Y^\bullet_{\geq m+2}\to Y^\bullet\to Y^{m+1}[-m-1]$$ in $\D^-(\mod eRe)$.  And then we have the following triangle in $\D^-(\mod R)$:
$$\D^-(T_e)(Y^{m+1})[-m-2]\to \D^-(T_e)(Y^\bullet_{\geq m+2})\to\D^-(T_e)(Y^\bullet)\to \D^-(T_e)(Y^{m+1})[-m-1].$$
By the induction hypothesis, we see that both $\D^-(T_e)(Y^{m+1})[-m-2]$ and $\D^-(T_e)(Y^\bullet_{\geq m+2})$ lie in $\D^b(\mod R)_{fgp}$. Thus $\D^-(T_e)(Y^\bullet)\in\D^b(\mod R)_{fgp}$.

For the ``only if'' part, assume $\D^-(T_e):\D^-(\mod eRe)\to\D^-(\mod R)$ restricts to $\D^b(T_e)_{fgp}:\D^b(\mod eRe)_{fgp}\to \D^b(\mod R)_{fgp}$. Let $G\in\Gproj eRe$. Take a projective resolution $P_G^\bullet\to G$ of $G$. It follows that $\D^-(T_e)(G)\cong T_e(P_G^\bullet)$, this should be a complex in $\D^b(\mod R)_{fgp}$. Thus $T_e(P_G^\bullet)$ has finite cohomology and then $\Tor_i^{eRe}(Re,G)=0$ for $i$ sufficiently large.
\end{proof}

Denote by $\D^b(\mod R)_{\mod R/ReR}$ the subcategory of $\D^b(\mod R)$ consisting of complexes with cohomology in $\mod R/ReR$. It is not hard to see $\D^b(\mod R)_{\mod R/ReR}$ is a thick subcategory of $\D^b(\mod R)$ generated by $\mod R/ReR$.
Inspired by \cite[Theorem 5.2]{PSS}, we get the following main result of this paper.

\begin{thm}\label{thm:1.1} Let $R$ be an Artin algebra and $e$ an idempotent of $R$. Assume that $\Tor_i^{eRe}(Re,G)=0$ for any $G\in\Gproj eRe$ and $i$ sufficiently large. Then the Schur functor $S_e$ induces a triangle-equivalence of Gorenstein defect categories $\D_{def}(R)\simeq\D_{def}(eRe)$ if and only if the following conditions are satisfied:
\begin{enumerate}
\item[(C1)] $\Gpd_{eRe}S_e(F)<\infty$ for any $F\in\Gproj R$.
\item[(C2)] $\Gpd_{R}T_e(G)<\infty$ for any $G\in\Gproj eRe$.
\item[(C3)] The idempotent $e$ is Gorenstein singularly-complete.
\end{enumerate}
\end{thm}

\begin{proof} For the ``if'' part. Since $0\to\mod R/ReR\to\mod R\to\mod eRe\to0$ is an exact sequence of module categories by Lemma \ref{lem:3.1}, it follows from \cite[Theorem 3.2]{M} that $\D^b(S_e):\D^b(\mod R)\to \D^b(\mod eRe)$ induces a triangle-equivalence $$\overline{\D^b(S_e)}:\D^b(\mod R)/\D^b(\mod R)_{\mod R/ReR}\to \D^b(\mod eRe).$$ Notice that $e$ is Gorenstein singularly-complete, thus any $R/ReR$-module has finite Gorenstein projective dimension as an $R$-module. As $\D^b(\mod R)_{\mod R/ReR}$ is generated by $\mod R/ReR$, one has $\D^b(\mod R)_{\mod R/ReR}\subseteq\D^b(\mod R)_{fgp}$. Then by Theorem \ref{thm:A.5}, we obtain
$$\D_{def}(R)\simeq(\D^b(\mod R)/\D^b(\mod R)_{\mod R/ReR})/(\D^b(\mod R)_{fgp}/\D^b(\mod R)_{\mod R/ReR}).$$

Since $\Gpd_{eRe}S_e(F)<\infty$ for any $F\in\Gproj R$, by Lemma \ref{lem:3.2} we have $\D^b(S_e)(\D^b(\mod R)_{fgp})\subseteq\D^b(\mod eRe)_{fgp}$. Therefore, $\overline{\D^b(S_e)}(\D^b(\mod R)_{fgp}/\D^b(\mod R)_{\mod R/ReR})\subseteq\D^b(\mod eRe)_{fgp}$.  This implies the following exact commutative diagram:
$$\xymatrix@C=3ex{0\ar[r]&{\D^b(\mod R)_{fgp}/\D^b(\mod R)}_{\mod R/ReR}\ar[r]\ar[d]^{\overline{\D^b(S_e)}_{fgp}}
&{\D^b(\mod R)/\D^b(\mod R)}_{\mod R/ReR}\ar[r]\ar[d]_{\simeq}^{\overline{\D^b(S_e)}}
& \D_{def}(R)\ar[d]^{\D_{def}(S_e)}\ar[r]&0\\
0\ar[r]&{\D^b(\mod eRe)_{fgp}}\ar[r]
&{\D^b(\mod eRe)}\ar[r]
&{\D_{def}(eRe)}\ar[r]&0
,}$$
where $\overline{\D^b(S_e)}_{fgp}$ is the restriction of $\overline{\D^b(S_e)}$. Since $\overline{\D^b(S_e)}$ is fully faithful, so is $\overline{\D^b(S_e)}_{fgp}$.

Next, we will show $\overline{\D^b(S_e)}_{fgp}:\D^b(\mod R)_{fgp}/\D^b(\mod R)_{\mod R/ReR}\to\D^b(\mod eRe)_{fgp}$ is essentially surjective (or dense). By \cite[Proposition 2.1]{CPS2} we have the following diagram:
$$\xymatrix{\D^-(\mod R)\ar@/_1pc/[r]_{\D^-(S_e)} &\D^-(\mod eRe)\ar@/_1pc/[l]_{\D^-(T_e)}, }$$
such that $(\D^-(T_e),\D^-(S_e))$ is an adjoint pair with $\D^-(T_e)$ fully faithful. In view of Lemma \ref{lem:3.2}, this diagram restricts to the following diagram:
$$\xymatrix{\D^b(\mod R)_{fgp}\ar@/_1pc/[r]_{\D^b(S_e)_{fgp}} &\D^b(\mod eRe)_{fgp}\ar@/_1pc/[l]_{\D^b(T_e)_{fgp}}, }$$
such that $(\D^b(T_e)_{fgp},\D^b(S_e)_{fgp})$ is an adjoint pair with $\D^b(T_e)_{fgp}$ fully faithful. Then for any $Y^\bullet\in\D^b(\mod eRe)_{fgp}$, denote by $X^\bullet=\D^b(T_e)_{fgp}(Y^\bullet)$. It follows that $Y^\bullet\cong\D^b(S_e)_{fgp}\D^b(T_e)_{fgp}(Y^\bullet)\cong\D^b(S_e)_{fgp}(X^\bullet)$. Then $\D^b(S_e)_{fgp}:\D^b(\mod R)_{fgp}\to \D^b(\mod eRe)_{fgp}$ is dense and hence $\overline{\D^b(S_e)}_{fgp}:\D^b(\mod R)_{fgp}/\D^b(\mod R)_{\mod R/ReR}\to\D^b(\mod eRe)_{fgp}$ is dense. To conclude, $\overline{\D^b(S_e)}_{fgp}$ is a triangle-equivalence. Therefore, we infer that $\D_{def}(S_e):\D_{def}(R)\to\D_{def}(eRe)$ is a triangle-equivalence from the above exact commutative diagram.

Conversely, let $F\in\Gproj R$. It follows from Proposition \ref{prop:A.4} that $F$ is zero in $\D_{def}(R)$. Since $\D_{def}(S_e):\D_{def}(R)\to\D_{def}(eRe)$ is a triangle-equivalence, we get $\D_{def}(S_e)(F)=S_e(F)$ is zero in $\D_{def}(\mod eRe)$ and hence $S_e(F)\in\D^b(\mod eRe)_{fgp}$. Therefore from Remark \ref{rem:2.7} we have $\Gpd_{eRe}S_e(F)<\infty$ and (C1) follows. To get (C2) and (C3), for any $G\in\Gproj eRe$ and $M\in\mod R/ReR$, we will show both $M$ and $T_e(G)$ have finite Gorenstein projective dimension as $R$-modules. Following Lemma \ref{lem:3.1}, we obtain $S_e(M)=0$ and $G\cong S_e T_e(G)$ and hence both $S_e(M)$ and $S_e T_e(G)$ are zero in $\D_{def}(eRe)$. Notice that $\D_{def}(S_e):\D_{def}(R)\to\D_{def}(eRe)$ is an equivalence, we obtain that  both $M$ and $T_e(G)$ are zero in $\D_{def}(R)$.
Therefore by the foregoing proof, we conclude that both $M$ and $T_e(G)$ have finite Gorenstein projective dimension as $R$-modules as desired.
\end{proof}

\begin{cor}\label{cor:3.3} Let $R$ be an Artin algebra and $e$ an idempotent of $R$. Assume that $\Tor_i^{eRe}(Re,G)=0$ for any $G\in\Gproj eRe$ and $i$ sufficiently large. Then the Schur functor $S_e$ induces the following exact commutative diagram:
$$\xymatrix{0\ar[r]&\underline{\Gproj R}\ar[r]\ar[d]
&\D_{sg}(R)\ar[r]\ar[d]
& \D_{def}(R)\ar[d]\ar[r]&0\\
0\ar[r]&\underline{\Gproj eRe}\ar[r]
&\D_{sg}(eRe)\ar[r]
&\D_{def}(eRe)\ar[r]&0
}$$
with all vertical functors triangle-equivalences if and only if the following conditions are satisfied:
\begin{enumerate}
\item[(C1)] $\Gpd_{eRe}S_e(F)<\infty$ for any $F\in\Gproj R$.
\item[(C2)] $\Gpd_{R}T_e(G)<\infty$ for any $G\in\Gproj eRe$.
\item[(C3)] The idempotent $e$ is singularly-complete and $\pd_{eRe}eR<\infty$.
\end{enumerate}
\end{cor}
\begin{proof} The ``only if'' part follows directly from Theorem \ref{thm:1.1} and \cite[Corollary 5.4]{PSS}, we will prove the ``if'' part. Combine \cite[Corollary 5.4]{PSS} with Theorem \ref{thm:1.1}, the Schur functor induces a singular equivalence $\D_{sg}(S_e):\D_{sg}(R)\to\D_{sg}(eRe)$ and a triangle-equivalence of Gorenstein defect categories $\D_{def}(S_e):\D_{def}(R)\to\D_{def}(eRe)$. Then we have the following exact commutative diagram:
$$\xymatrix{0\ar[r]&\underline{\Gproj R}\ar[r]
&\D_{sg}(R)\ar[r]\ar[d]^{\D_{sg}(S_e)}
& \D_{def}(R)\ar[d]^{\D_{def}(S_e)}\ar[r]&0\\
0\ar[r]&\underline{\Gproj eRe}\ar[r]
&\D_{sg}(eRe)\ar[r]
&\D_{def}(eRe)\ar[r]&0
.}$$
Hence it is not hard to see that $\D_{sg}(S_e):\D_{sg}(R)\to\D_{sg}(eRe)$ restricts to a triangle-equivalence $\underline{\Gproj R}\simeq\underline{\Gproj eRe}$ and then we get the desired result.
\end{proof}

\section{Applications in triangular matrix algebras}\label{app}

In this section, we will deal with the triangular matrix algebra $T=\left(
                                                                                 \begin{array}{cc}
                                                                                   A & M \\
                                                                                   0 & B \\
                                                                                 \end{array}
                                                                               \right)$, where the corner algebras $A$ and $B$ are Artin algebras and ${_A}M_B$ is an $A$-$B$-bimodule.

Recall that a left $T$-module is identified with a triple $(X, Y, \phi)$, where $X\in \mod A$, $Y\in \mod B$ and $\phi:M\otimes_B Y\to X$ ia an $A$-morphism. If there is no possible confusion, we shall omit the morphism $\phi$ and write $(X, Y)$ for short. For example we write $(M\otimes_SY, Y)$ for the $T$-module $(M\otimes_SY, Y, id)$.
A $T$-morphism
$(X, Y, \phi)\to(X', Y', \phi')$
will be identified with a pair
$(f, g)$,
where $f\in\Hom_A(X,X')$ and $g\in\Hom_B(Y,Y')$, such that the following diagram
$$\xymatrix{M\otimes_B Y\ar[r]^{\ \ \phi}\ar[d]_{1\otimes g} & X\ar[d]_f\\
M\otimes_B Y'\ar[r]^{\ \ \phi'} &X'
}$$ is commutative.

A sequence $0\to(X_1, Y_1, \phi_1)\xrightarrow{(f_1, g_1)}(X_2, Y_2, \phi_2)\xrightarrow{(f_2, g_2)}(X_3, Y_3, \phi_3)\to0$
in $\mod T$ is exact if and only if $0\to X_1\xrightarrow{f_1}X_2\xrightarrow{f_2}X_3\to0$ and $0\to Y_1\xrightarrow{g_1}Y_2\xrightarrow{g_2}Y_3\to0$ are exact in $\mod A$ and $\mod B$ respectively.
Indecomposable projective $T$-modules are exactly $(P, 0)$ and $(M\otimes Q, Q)$, where $P$ runs over indecomposable projective $A$-modules, and $Q$ runs over indecomposable projective $B$-modules.
%
We refer the reader to \cite{AS,ARS,G} for more details.

Recall from \cite{ZP} that $_AM_B$ is {\it compatible} if $M\otimes_B-$ carries every acyclic complex of projective $B$-modules to acyclic $A$-complex and $M\in(\Gproj A)^\perp$.
If $_AM_B$ is compatible, it is not hard to see $\Tor_i^B(M,G)=0$ for any $G\in\Gproj B$ and $i\geq1$.
%
%

\begin{lem}(\cite[Theorem 1.4]{ZP})\label{lem:4.1} Let $T=\left(
                                                                                 \begin{array}{cc}
                                                                                   A & M \\
                                                                                   0 & B \\
                                                                                 \end{array}
                                                                               \right)$ be a triangular matrix algebra with $_AM_B$ compatible.
Then $(X, Y, \phi)\in\Gproj T$ if and only if $Y\in\Gproj B$ and $\phi:M\otimes_B Y\to X$ is an injective $A$-morphism with $\Coker\phi\in\Gproj A$.
\end{lem}
As a consequence, we have the following.
\begin{cor}\label{cor:4.3} Let $T=\left(
                                                                                 \begin{array}{cc}
                                                                                   A & M \\
                                                                                   0 & B \\
                                                                                 \end{array}
                                                                               \right)$ be a triangular matrix algebra with $_AM_B$ compatible. The following hold true.

                                                                               (1) $\Gpd_T(X, 0)=\Gpd_AX$.

                                                                                                                                                        (2) Assume that $\Gpd_AM\otimes_BG<\infty$ for any $G\in\Gproj B$.
                                                                                                                                                        Then $\Gpd_T(0,Y)<\infty$ if and only if $\Gpd_BY<\infty$.
\end{cor}
\begin{proof} (1) is clear and we only prove (2).
For the ``only if'' part, let $\Gpd_T(0,Y)=n$ for some integer $n\geq 0$. Then there exists an exact sequence $$0\to(P_n, Q_n, {\phi_n})\to\cdots\to(P_1, Q_1, {\phi_1})\to(P_0,Q_0, {\phi_0})\to(0, Y)\to0$$ in $\mod T$ with each $(P_i, Q_i, {\phi_i})\in\Gproj T$. It follows that $$0\to Q_n\to\cdots\to Q_1\to Q_0\to Y\to0$$ is an exact sequence in $\mod B$. Since each $(P_i, Q_i, {\phi_i})\in\Gproj T$, it follows from Lemma \ref{lem:4.1} that each $Q_i\in\Gproj B$. Therefore, $\Gpd_BY\leq n$.

For the ``if'' part, we first claim $\Gpd_T(0,G)<\infty$ for any $G\in\Gproj B$. To get this, let $G\in\Gproj B$. Consider the following exact sequence of $T$-modules:
$$0\to(M\otimes_BG, 0)\to(M\otimes_BG, G)\to(0, G)\to0.$$
By assumption, $\Gpd_AM\otimes_BG<\infty$. As a consequence of (1), we get $\Gpd_T(M\otimes_BG, 0)<\infty$. Notice that $(M\otimes_BG, G)\in\Gproj T$, we obtain $\Gpd_T(0,G)<\infty$ and the claim follows.

Now assume $\Gpd_BY=m$ for some integer $m\geq0$.
Take a Gorenstein projective resolution $0\to G_m\to\cdots\to G_1\to G_0\to Y\to0$ of $Y$. We have $0\to (0,G_m)\to\cdots\to (0,G_1)\to (0,G_0)\to (0,Y)\to0$ is exact in $\mod T$. By the claim, one has $\pd_T(0,G_i)<\infty$ for every $i$. Therefore, we conclude that $\pd_T(0,Y)<\infty$.
\end{proof}

Let $e_A=\left(
                           \begin{array}{cc}
                             1 & 0 \\
                             0 & 0 \\
                           \end{array}
                         \right)$ and $e_B=\left(
                           \begin{array}{cc}
                             0 & 0 \\
                             0 & 1 \\
                           \end{array}
                         \right)$ be the idempotents of $T$. It is known that $A\cong e_ATe_A\cong T/Te_BT$ and $B\cong e_BTe_B\cong T/Te_AT$ as algebras.
                          Denote by $S_{e_A}$ and $S_{e_B}$ the Schur functors associative to $e_A$ and $e_B$ respectively.

                          We have the following observation.

\begin{lem}\label{lem:4.3} The following statements hold true:
\begin{enumerate}
\item[(1)] $\Tor_i^{A}(Te_A,F)=0$ for any $F\in\Gproj A$ and $i\geq1$.
\item[(2)] If $_AM_B$ is compatible, then $\Tor_i^{B}(Te_B,G)=0$ for any $G\in\Gproj B$ and $i\geq1$.
\end{enumerate}
\end{lem}
\begin{proof} (1) Since $Te_A\cong A$ as right $A$-modules, this assertion is trivial.

(2) Note that $Te_B\cong M_B\oplus B_B$ is an isomorphism of right $B$-modules.
Since $_AM_B$ is compatible, $\Tor_i^B(M,G)=0$ for any $G\in\Gproj B$ and $i\geq1$. And hence $\Tor_i^B(Te_B,G)=0$ for any $G\in\Gproj B$ and $i\geq1$.
\end{proof}

To take $R=T$ and $e=e_A$ as in Theorem \ref{thm:1.1} and Corollary \ref{cor:3.3}, we get the following
\begin{thm}\label{thm:1.3} Let $T=\left(
                                                                                 \begin{array}{cc}
                                                                                   A & M \\
                                                                                   0 & B \\
                                                                                 \end{array}
                                                                               \right)$ be a triangular matrix algebra with $_AM_B$ compatible. The following statements hold true:
\begin{enumerate}
\item[(1)] The Schur functor $S_{e_A}$ induces a triangle-equivalence $\D_{def}(T)\simeq\D_{def}(A)$ if and only if $B$ is Gorenstein and $\Gpd_AM\otimes_BY<\infty$ for any $Y\in\Gproj B$.
In this case, we have a singular equivalence $\D_{sg}(\Aus(T))\simeq\D_{sg}(\Aus(A))$ between the relative Auslander algebras $\Aus(T)$ and $\Aus(A)$ provided that $T$ is CM-finite.
\item[(2)] The Schur functor $S_{e_A}$ induces the following exact commutative diagram:
$$\xymatrix{0\ar[r]&\underline{\Gproj T}\ar[r]\ar[d]
&\D_{sg}(T)\ar[r]\ar[d]
& \D_{def}(T)\ar[d]\ar[r]&0\\
0\ar[r]&\underline{\Gproj A}\ar[r]
&\D_{sg}(A)\ar[r]
&\D_{def}(A)\ar[r]&0
}$$
with all vertical functors triangle-equivalences if and only if $B$ has finite global dimension, $\pd_AM<\infty$ and $\Gpd_AM\otimes_BY<\infty$ for any $Y\in\Gproj B$. In this case $T$ is Gorenstein (resp. CM-free) if and only if so is $A$.
\end{enumerate}
\end{thm}

\begin{proof} Since $e_ATe_A\cong A$, we will prove by replacing $R$ with $T$ and $e$ with $e_A$ in Theorem \ref{thm:1.1} and Corollary \ref{cor:3.3} respectively. By Lemma \ref{lem:4.3} (1),  $\Tor_i^{A}(Te_A,F)=0$ for any $F\in\Gproj A$ and $i\geq1$ .

(1) For the ``only if'' part, assume that $S_{e_A}$ induces a triangle-equivalence $\D_{def}(T)\simeq\D_{def}(A)$. Then conditions (C1)-(C3) in Theorem \ref{thm:1.1} are satisfied. For any $Y\in\Gproj B$, it follows from Lemma \ref{lem:4.1} that $(M\otimes_BY, Y)\in\Gproj T$. Since $S_{e_A}(M\otimes_BY, Y)\cong  M\otimes_BY$, we infer $\Gpd_AM\otimes_BY<\infty$ from (C1). Now let $Z\in\mod B$.
Note that $Z$ viewed as a $T$-module is isomorphic to $(0, Z)$. We infer that $\Gpd_T(0, Z)<\infty$ from (C3). In view of Corollary \ref{cor:4.3} (2), we obtain $\Gpd_BZ<\infty$. Then it follows from \cite[Theorem]{Hos} that $B$ is Gorenstein.

For the ``if'' part, assume that $B$ is Gorenstein and $\Gpd_AM\otimes_BY<\infty$ for any $Y\in\Gproj B$. In view of Theorem \ref{thm:1.1}, it suffices to show conditions (C1)-(C3) in Theorem \ref{thm:1.1} are satisfied. Let $(X, Y, \phi)\in\Gproj T$. It follows from Lemma \ref{lem:4.1} that $Y\in\Gproj B$ and $\phi:M\otimes_B Y\to X$ is an injective $A$-morphism with $\Coker\phi\in\Gproj A$. Now consider the following exact sequence of $A$-modules:
$$0\to M\otimes_B Y\xrightarrow{\phi} X \to\Coker\phi\to0.$$
Because $\Coker\phi\in\Gproj A$ and $\Gpd_AM\otimes_BY<\infty$, we have that $\Gpd_AX<\infty$. Notice that $X\cong S_{e_A}(X, Y, \phi)$, we have (C1) follows. Now for any $F\in\Gproj A$, we have $\Hom_T(T_{e_A}(F),(X', Y' , \phi'))\cong\Hom_A(F,S_{e_A}(X', Y' , \phi'))\cong\Hom_A(F,X')\cong\Hom_T((F, 0),(X', Y' , \phi'))$ for any $(X', Y' , \phi')\in\mod T$. By Yoneda Lemma, we have $T_{e_A}(F)\cong(F, 0)$.
Following Lemma \ref{lem:4.1}, $T_{e_A}(F)\in\Gproj T$ and then condition (C2) follows. For (C3), let $Z\in\mod B$. Notice that $Z$ viewed as a $T$-module is isomorphic to $(0, Z)$, we will show $\Gpd_T(0, Z)<\infty$.
Since $B$ is Gorenstein, one gets $\Gpd_BZ<\infty$. Hence it follows from Corollary \ref{cor:4.3} (2) that $\Gpd_T(0, Z)<\infty$.

Following Lemma \ref{lem:4.1}, it is not hard to see if $T$ is CM-finite, then so is $A$.  By Remark \ref{rem:2.4}, we infer a singular equivalence $\D_{sg}(\Aus(T))\simeq\D_{sg}(\Aus(A))$ from the triangle-equivalence $\D_{def}(T)\simeq\D_{def}(A)$.

(2) For the ``only if'' part, assume that $S_{e_A}$ induces such an exact commutative diagram. It follows from Corollary \ref{cor:3.3} that $e_A$ is singularly-complete and $\pd_Ae_AT<\infty$. Notice that $e_AT\cong A\oplus M$ is an isomorphism of $A$-modules, we obtain $\pd_AM<\infty$. Now for any $Z\in\mod B$, since $e_A$ is singularly-complete, it follows that $\pd_T(0, Z)<\infty$. We may assume $\pd_T(0, Z)=n$ for some integer $n\geq0$. Take a projective resolution $0\to(P_n, Q_n, {\phi_n})\to\cdots\to(P_1, Q_1, {\phi_1})\to(P_0,Q_0, {\phi_0})\to(0, Z)\to0$ of $(0, Z)$. It follows that $0\to Q_n\to\cdots\to Q_1\to Q_0\to Z\to0$ is a projective resolution of $Z$ and then $\pd_BZ<\infty$. This implies that $B$ has finite global dimension. Note that $S_{e_A}$ induces a triangle-equivalence $\D_{def}(T)\simeq\D_{def}(A)$. We have that $\Gpd_AM\otimes_BY<\infty$ for any $Y\in\Gproj B$ as a consequence of (1).

For the ``if'' part, assume that $B$ has finite global dimension, $\pd_AM<\infty$ and $\Gpd_AM\otimes_BY<\infty$ for any $Y\in\Gproj B$. By the foregoing proof we know that conditions (C1) and (C2) in Corollary \ref{cor:3.3} are satisfied. To get this assertion, in view of Corollary \ref{cor:3.3}, it suffices to show $e_A$ is singularly-complete and $\pd_Ae_AT<\infty$. Since $\pd_AM<\infty$, we infer $\pd_Ae_AT<\infty$ from the isomorphism $e_AT\cong A\oplus M$. Now let $Z\in\mod B$. Because $B$ has finite global dimension, one has $\pd_BZ<\infty$. If $Z\in\proj B$, we get that $\pd_AM\otimes_BZ<\infty$, since $\pd_AM<\infty$. It follows from \cite[Lemma 3.1 (1)]{C1} that $\pd_T(M\otimes_BZ,0)<\infty$. Notice that $(M\otimes_BZ,Z)$ is projective, and hence we infer $\pd_T(0, Z)<\infty$ from the short exact sequence $0\to(M\otimes_BZ,0)\to(M\otimes_BZ,Z)\to(0,Z)\to0$. Now assume $\pd_BZ=n$ for some integer $n>0$. Take a projective resolution $0\to P_n\to\cdots\to P_1\to P_0\to Z\to0$ of $Z$. We have $0\to (0,P_n)\to\cdots\to (0,P_1)\to (0,P_0)\to (0,Z)\to0$ is exact in $\mod T$. Note that $\pd_T(0,P_i)<\infty$ for every $i$. We conclude that $\pd_T(0,Z)<\infty$ and hence $e_A$ is singularly-complete.

Note that $T$ is Gorenstein if and only if $\D_{def}(T)=0$ (Corollary \ref{cor:A.6}); while $T$ is CM-free if and only $\underline{\Gproj T}=0$. Then we infer that $T$ is Gorenstein (resp. CM-free) if and only if so is $A$ from
such exact commutative diagram.
\end{proof}

\begin{exa}\label{exa:4.5}{\rm Let $k$ be a field and $Q$ the following quiver:
$$\xymatrix{
1\ar[r]^\alpha &2 \ar@/^/[r]^{\beta}&3\ar[r]^{\delta}\ar@/^/[l]^{\gamma} &4.
}$$
Consider the $k$-algebra $T=kQ/I$, where $I$ is generated by $\beta\gamma$, $\gamma\beta$ and $\delta\beta$. Let $e_i$ be the idempotent corresponding to the vertex $i$ and put $e=e_2+e_3+e_4$.
Denote by $A=eTe$ and $B=e_1Te_1$. Then $T=\left(
                                                                                 \begin{array}{cc}
                                                                                   A & M \\
                                                                                   0 & B \\
                                                                                 \end{array}
                                                                               \right)$ with $M=eTe_1$. Clearly $B=k$, and hence every $B$-module (left or right) is projective. It is easy to verify $_AM$ is projective and then $\Gpd_AM\otimes_BY<\infty$ for any $Y\in\mod B$. Following Theorem \ref{thm:1.3}, we get the following exact commutative diagram:
$$\xymatrix{0\ar[r]&\underline{\Gproj T}\ar[r]\ar[d]
&\D_{sg}(T)\ar[r]\ar[d]
& \D_{def}(T)\ar[d]\ar[r]&0\\
0\ar[r]&\underline{\Gproj A}\ar[r]
&\D_{sg}(A)\ar[r]
&\D_{def}(A)\ar[r]&0
}$$
with all vertical functors triangle-equivalences.  Note that $A$ is of radical square zero but not self-injective. Following \cite{C2'} $A$ is CM-free, and then so is $T$.
Thus $\underline{\Gproj T}$ and $\underline{\Gproj A}$ are trivial. Hence we get triangle-equivalences
$\D_{def}(T)=\D_{sg}(T)\simeq\D_{sg}(A)$($=\D_{def}(A)$).
}\end{exa}

To take $R=T$ and $e=e_B$ as in Theorem \ref{thm:1.1} and Corollary \ref{cor:3.3}, we get the following
\begin{thm}\label{thm:1.2} Let $T=\left(
                                                                                 \begin{array}{cc}
                                                                                   A & M \\
                                                                                   0 & B \\
                                                                                 \end{array}
                                                                               \right)$ be a triangular matrix algebra with $_AM_B$ compatible. The following statements hold true:
\begin{enumerate}
\item[(1)] The Schur functor $S_{e_B}$ induces a triangle-equivalence $\D_{def}(T)\simeq\D_{def}(B)$ if and only if $A$ is Gorenstein.
In this case, we have a singular equivalence $\D_{sg}(\Aus(T))\simeq\D_{sg}(\Aus(B))$ between the relative Auslander algebras $\Aus(T)$ and $\Aus(B)$ provided that $T$ is CM-finite.
\item[(2)] The Schur functor $S_{e_B}$ induces the following exact commutative diagram:
$$\xymatrix{0\ar[r]&\underline{\Gproj T}\ar[r]\ar[d]
&\D_{sg}(T)\ar[r]\ar[d]
& \D_{def}(T)\ar[d]\ar[r]&0\\
0\ar[r]&\underline{\Gproj B}\ar[r]
&\D_{sg}(B)\ar[r]
&\D_{def}(B)\ar[r]&0
}$$
with all vertical functors triangle-equivalences if and only if $A$ has finite global dimension. In this case, $T$ is Gorenstein (resp. CM-free) if and only if so is $B$.
\end{enumerate}
\end{thm}

\begin{proof} Notice that $e_BTe_B\cong B$, we will prove by replacing $R$ with $T$ and $e$ with $e_B$ in Theorem \ref{thm:1.1} and Corollary \ref{cor:3.3} respectively. Since $_AM_B$ is compatible, by Lemma \ref{lem:4.3}, we get $\Tor_i^{B}(Te_B,G)=0$ for any $G\in\Gproj B$ and $i\geq1$.

(1) Let $(X, Y , \phi)\in\Gproj T$, it follows that $S_{e_B}(X, Y , \phi)\cong Y$. Following Lemma \ref{lem:4.1}, we have $Y\in\Gproj B$ and then condition (C1) in Theorem \ref{thm:1.1} follows. Now for any $G\in\Gproj B$, we have $\Hom_T(T_{e_B}(G),(X', Y' , \phi'))\cong\Hom_B(G,S_{e_B}(X', Y' , \phi'))\cong\Hom_B(G,Y')\cong\Hom_T((M\otimes_BG, G),(X', Y' , \phi'))$ for every $(X', Y' , \phi')\in\mod T$.  By Yoneda Lemma, we have $T_{e_B}(G)\cong(M\otimes_BG, G)$. Then $T_{e_B}(G)$ is a Gorenstein projective $T$-module by Lemma \ref{lem:4.1} and hence condition (C2) in Theorem \ref{thm:1.1} follows.
Thus to get the desired assertion, in view of Theorem \ref{thm:1.1}, it suffices to show $A$ is Gorenstein if and only if $e_B$ is Gorenstein singularly-complete. Note that every $A$-module $W\in\mod A$ viewed as a $T$-module is isomorphic to $(W, 0)$.
Combine Corollary \ref{cor:4.3} (1) with \cite[Theorem]{Hos}, we conclude that $A$ is Gorenstein if and only if every $A$-module has finite Gorenstein projective dimension as a $T$-module if and only if $e_B$ is Gorenstein singularly-complete.

In this case, assume that $T$ is CM-finite. It is not hard to see $B$ is also CM-finite. By Remark \ref{rem:2.4}, we infer a singular equivalence $\D_{sg}(\Aus(T))\simeq\D_{sg}(\Aus(B))$ from the triangle-equivalence $\D_{def}(T)\simeq\D_{def}(B)$.

(2) By the foregoing proof, we get conditions (C1) and (C2) in Corollary \ref{cor:3.3} are always satisfied. Notice that there exists an isomorphism of $B$-modules $e_BT\cong B$, we get $e_BT$ is projective. In view of Corollary \ref{cor:3.3}, it suffices to show $A$ has finite global dimension if and only if $e_B$ is Gorenstein singularly-complete. Note that every $A$-module $W\in\mod A$ viewed as a $T$-module is isomorphic to $(W, 0)$.
It follows from \cite[Lemma 3.1 (1)]{C1} that $\pd_T(W, 0)=\pd_AW$. Therefore, $A$ has finite global dimension if and only if $\pd_T(W, 0)<\infty$ if and only if $e_B$ is singularly-complete.

In this case, by a similar argument as that in the proof of Theorem \ref{thm:1.3} (2), we conclude that $T$ is Gorenstein (resp. CM-free) if and only if so is $B$.
\end{proof}

\begin{exa}\label{exa:4.7} {\rm (1) Let $k$ be a field and $Q$ the following quiver:
$$\xymatrix{1\ar[r]^\alpha &2&\\
3\ar[u]^{\gamma} \ar@/^/[r]^{\beta}&4\ar[u]_{\delta}\ar[r]^{\theta}\ar@/^/[l]^{\beta'} &5.
}$$
Consider the $k$-algebra $T=kQ/I$, where $I$ is generated by $\beta'\beta$, $\beta\beta'$, $\theta\beta$ and $\alpha\gamma-\delta\beta$. Let $e_i$ be the idempotent corresponding to the vertex $i$ and put $e=e_3+e_4+e_5$.
Denote by $A=(1-e)T(1-e)$ and $B=eTe$. Then $T=\left(
                                                                                 \begin{array}{cc}
                                                                                   A & M \\
                                                                                   0 & B \\
                                                                                 \end{array}
                                                                               \right)$ with $M=(1-e)Te$. It is easy to see the global dimension of $A$ is $1$ and $M_B$ is projective. Then $_AM_B$ is compatible.
                                                                                Following Theorem \ref{thm:1.2}, we get the following exact commutative diagram:
$$\xymatrix{0\ar[r]&\underline{\Gproj T}\ar[r]\ar[d]
&\D_{sg}(T)\ar[r]\ar[d]
& \D_{def}(T)\ar[d]\ar[r]&0\\
0\ar[r]&\underline{\Gproj B}\ar[r]
&\D_{sg}(B)\ar[r]
&\D_{def}(B)\ar[r]&0
}$$
with all vertical functors triangle-equivalences. Notice that $B$ is CM-free, then so is $T$. Thus $\underline{\Gproj T}$ and $\underline{\Gproj B}$ are trivial. Hence we get triangle-equivalences $\D_{def}(T)=\D_{sg}(T)\simeq\D_{sg}(B)$($=\D_{def}(B)$).

(2) Let $k$ be a field and $Q$ the following quiver:
$$\xymatrix{1\ar@/^/[r]^\alpha &2\ar@/^/[l]^{\alpha'}&\\
3\ar[u]^{\gamma} \ar@/^/[r]^{\beta}&4\ar[u]_{\delta}\ar[r]^{\theta}\ar@/^/[l]^{\beta'} &5.
}$$
Consider the $k$-algebra $T=kQ/I$, where $I$ is generated by $\alpha'\alpha$, $\alpha\alpha'$, $\beta'\beta$, $\beta\beta'$, $\theta\beta$, $\alpha\gamma-\delta\beta$ and $\alpha'\delta-\gamma\beta'$. Let $e_i$ be the idempotent corresponding to the vertex $i$ and put $e=e_3+e_4+e_5$.
$A=(1-e)T(1-e)$ and $B=eTe$. Then $T=\left(
                                                                                 \begin{array}{cc}
                                                                                   A & M \\
                                                                                   0 & B \\
                                                                                 \end{array}
                                                                               \right)$ with $M=(1-e)Te$. Clearly, $A$ is self-injective and $B$ is CM-free. It is easy to see $_AM$ and $M_B$ are projective and then $_AM_B$ is compatible. In view of Theorem \ref{thm:1.2} (1), we get a triangle-equivalence $\D_{def}(T)\simeq\D_{def}(B)=\D_{sg}(B)$.
                                                                               However, since the global dimension of $A$ is infinite, the Schur functor $S_e$ does not induce an exact commutative diagram as that in Theorem \ref{thm:1.2} (2). Thus we conclude $S_e$ induces neither a singular equivalence $\D_{sg}(T)\simeq\D_{sg}(B)$ nor a triangle-equivalence $\underline{\Gproj T}\simeq\underline{\Gproj B}$.
}\end{exa}

\bigskip {\bf Acknowledgements}
\bigskip

This research was partially supported by NSFC (Grant No. 11501257, 11626179, 11671069, 11701455, 11771212), Qing Lan
Project of Jiangsu Province, Shaanxi Province Basic Research Program of Natural Science (Grant No. 2017JQ1012) and Fundamental Research Funds for the Central Universities (Grant No. JB160703, 2452020182).


\begin{thebibliography}{101}
\bibitem{AS} J. Asadollahi, S. Salarian, On the vanishing of Ext over formal triangular matrix rings. Forum Math {\bf 18} (2006), 951--966.
\bibitem{ASS} I. Assem, D. Simson, A. Skowronski, {\it Elements of the representation theory of associative algebras. Vol. 1. Techniques
of representation theory}, London Mathematical Society Student Texts, 65. Cambridge University Press, Cambridge, 2006.
\bibitem{AB} M. Auslander, M. Bridger, {\it Stable module theory},
Memoirs Amer. Math. Soc. {\bf 94}, Amer. Math. Soc., Providence, RI, 1969.



\bibitem{ARS} M. Auslander, I. Reiten, S.O. Smal{\o}, {\it Representation Theory of Artin Algebras}, Cambridge University Press, 1995.

\bibitem{AFHd} L.L. Avramov, H.-B. Foxby, \emph{Homological dimensions of unbounded complexes}, J. Pure Appl. Algebra {\bf 71} (1991) 129--155.
\bibitem{AM} L.L. Avramov, A. Martsinkovsky, {\it Absolute, relative, and Tate cohomology of modules of finite Gorenstein dimension}, Proc. Lond. Math. Soc. {\bf 85} (2002), 393--440.


\bibitem{Be1} A. Beligiannis, {\it The homological theory of contravariantly finite subcategories:
Gorenstein categories, Auslander-Buchweitz contexts and (co-)stabilization}, Comm. Algebra {\bf 28} (2000), 4547--4596.


\bibitem{Be2} A. Beligiannis, {\it On algebras of finite Cohen-Macaulay type}, Adv. Math. {\bf 226} (2011), 1973--2019.

\bibitem{BJO} P.A. Bergh, D.A. J{\o}rgensen, S. Oppermann, {\it The Gorenstein defect category}, Quart. J. Math. {\bf 66}(2) (2015), 459--471.



\bibitem{Bu} R.O. Buchweitz, {\it Maximal Cohen-Macaulay Modules and Tate Cohomology over Gorenstein Rings}, Unpublished manuscript, 1986.


\bibitem{C1} X.W. Chen, {\it Singularity Categories, Schur functors and triangular matrix tings}, Algebra Represent. Theor. {\bf 12} (2009), 181--191.


\bibitem{C2} X.W. Chen, {\it Relative singularity categories and Gorenstein-projective modules}, Math. Nachr. {\bf 284} (2011), 199--212.
\bibitem{C2'} X.W. Chen, {\it Algebras with radical square zero are either self-injective or CM-free}, Proc. Amer. Math. Soc. {\bf 140}(1) (2012), 93--98.
\bibitem{C3} X.W. Chen, {\it Singular equivalences induced by homological epimorphisms}, Proc. Amer. Math. Soc. {\bf 142}(8) (2014), 2633--2640.
\bibitem{C4} X.W. Chen, {\it Singular equivalences of trivial extensions}, Comm. Algebra {\bf 44}(5)  (2016), 1961--1970.
\bibitem{CZ} X.W. Chen, P. Zhang, {\it Quotient triangulated categories}, Manuscripta Math. {\bf 123} (2007), 167--183.

\bibitem{CPS2} E. Cline, B. Parshall, L. Scott, {\it Finite dimensional algebras and hight weightest categories}, J. Reine. Angew. Math. {\bf 391} (1988), 85--99.
%
\bibitem{G} J.A. Green, {\it Polynomial Representations of $GL_n$}, Lecture Notes in Math, vol. 830. Springer, New York (1980).
\bibitem{H1} D. Happel, {\it Triangulated Categories in Representation Theory of Finite Dimensional Algebras},
Lond. Math. Soc. Lect. Notes Ser. {\bf 119}, Cambridge Univ. Press, Cambridge, 1988.
\bibitem{H2} D. Happel, {\it On Gorenstein Algebras}, in: Representation theory of finite groups and
finite-dimensional algebras (Proc. Conf. at Bielefeld, 1991), Progress in Math. {\bf 95}, Birkh\"{a}user, Basel, 1991, pp.389--404.

\bibitem{Ho} H. Holm, {\it Gorenstein homological dimensions}, J. Pure Appl. Algebra {\bf 189} (2004), 167--193.


\bibitem{Hos} M. Hoshino, {\it Algebras of finite self-injective dimension}, Proc. Amer. Math. Soc. {\bf 112} (1991), 619--622.





\bibitem{KZ} F. Kong, P. Zhang, {\it From CM-finite to CM-free}, J. Pure Appl. Algebra {\bf 220}(2) (2016), 782--801.

\bibitem{LL} P. Liu, M. Lu, {\it Recollements of singularity categories and monomorphism categories}, Comm, Algebra {\bf 43} (2015), 2443--2456.

\bibitem{L} M. Lu, {\it Gorenstein defect categories of triangular matrix algebras}, J. Algebra {\bf 480} (2017), 346--367.
\bibitem{M} J.I. Miyachi, {\it Localization of triangulated categories and derived categories}, J. Algebra {\bf 141} (1991), 463--483.


\bibitem{O} D. Orlov, {\it Triangulated categories of singularities and D-branes in Landau-Ginzburg
models}, Proc. Steklov Inst. Math. {\bf 246} (2004), 227--248.

\bibitem{P} C. Psaroudakis, {\it Homological theory of recollements of abelian categories}, J. Algebra {\bf 398} (2014), 63--110.
\bibitem{PSS} C. Psaroudakis, O. Skartsaterhagen, O. Solberg, {\it Gorenstein categories, singular equivalences
and finite generation of cohomology rings in recollements}, Trans. Amer. Math. Soc. Ser. B {\bf 1} (2014), 45--95.

\bibitem{R1} J. Rickard, {\it Derived categories and stable equivalence}, J. Pure Appl. Algebra {\bf 61} (1989), 303--317.



\bibitem{Vel} O. Veliche, {\it Gorenstein projective dimension for complexes}, Trans. Amer. Math. Soc. {\bf 358} (2006), 1257--1283.

\bibitem{ZP} P. Zhang, {\it Gorenstein-projective modules and symmetric recollements}, J. Algebra {\bf 388} (2013), 65--80.
\bibitem{ZH} Y.F. Zheng, Z.Y. Huang, {\it Triangulated equivalences involving Gorenstein projective modules}, Canad. Math. Bull. {\bf 60}(4) (2017), 879--890.
\bibitem{Z} S.J. Zhu, {\it Left homotopy theory and Buchweitz¡¯s theorem}, Master Thesis at SJTU, 2011.


\end{thebibliography}
\end{document}